\documentclass[12 pt]{article}
\usepackage[pdftex]{graphicx}
\usepackage{amsmath}
\usepackage{mathtools}
\usepackage{amsthm}
\usepackage{amsfonts}
\usepackage{amssymb}
\usepackage[margin=1.0in]{geometry}
\usepackage{tikz}
\usepackage{enumitem}
\usepackage{bm}
\usepackage{multicol}

\newtheorem{theorem}{Theorem}[section]
\newtheorem*{theorem*}{Theorem}
\newtheorem{cor}[theorem]{Corollary}
\newtheorem{lemma}[theorem]{Lemma}
\newtheorem{prop}[theorem]{Proposition}

\theoremstyle{definition}
\newtheorem{defin}[theorem]{Definition}

\theoremstyle{remark}

\renewcommand{\phi}{\varphi}
\newcommand{\op}[1]{\mathrm{op}(#1)}

\newcommand{\capdots}{\cap\cdots\cap}
\newcommand{\cupdots}{\cup\cdots\cup}

\begin{document}
\title{A direct solution to the Generic Point Problem}
\author{Andy Zucker}
\date{}
\maketitle

\begin{abstract}
We provide a new proof of a recent theorem of Ben-Yaacov, Melleray, and Tsankov. If $G$ is a Polish group and $X$ is a minimal, metrizable $G$-flow with all orbits meager, then the universal minimal flow $M(G)$ is non-metrizable. In particular, we show that given $X$ as above, the universal highly proximal extension of $X$ is non-metrizable.
\let\thefootnote\relax
\let\thefootnote\relax
\let\thefootnote\relax\footnote{The author was partially supported by NSF Grant no.\ DGE 1252522.}
\end{abstract}

\section{Introduction}
In this paper, we are concerned with actions of a topological group $G$ on a compact space $X$. All groups and spaces are assumed Hausdorff. A compact space $X$ equipped with a continuous $G$-action $a: G\times X\rightarrow X$ is called a \emph{$G$-flow}. The action $a$ is often suppressed in the notation, i.e.\ $gx$ is written for $a(g,x)$. A $G$-flow $X$ is called \emph{minimal} if every orbit is dense. It is a fact that every topological group $G$ admits a \emph{universal minimal flow} $M(G)$, a minimal flow which admits a $G$-map onto any other minimal flow. A \emph{$G$-map} is a continuous map respecting the $G$-action. The flow $M(G)$ is unique up to $G$-flow isomorphism. 

We can now recall the following theorem of Ben-Yaacov, Melleray, and Tsankov \cite{BYMT}.

\begin{theorem}
\label{MainTheorem}
Let $G$ be a Polish group, and let $M(G)$ be the universal minimal flow of $G$. If $M(G)$ is metrizable, then $M(G)$ has a comeager orbit. 
\end{theorem}

The question of whether or not metrizability of $M(G)$ was enough to guarantee a comeager orbit was first asked by Angel, Kechris, and Lyons \cite{AKL}. In \cite{Z}, the current author proved Theorem \ref{MainTheorem} in the case when $G$ is the automorphism group of a first-order structure. The proof given there used topological properties of the largest $G$-ambit $S(G)$ along with combinatorial reasoning about the structures. In \cite{BYMT}, the authors also use topological properties of $S(G)$, but the combinatorics is replaced by the following theorem due to Rosendal; see \cite{BYMT} for a proof.

\begin{theorem}
\label{RosendalThm}
Let $G$ be a Polish group acting continuously on a compact metric space $X$. Assume the action is topologically transitive. Then the following are equivalent.
\begin{enumerate}
\item
$G$ has a comeager orbit.
\item
For any open $1\in V\subseteq G$ and any open $B\subseteq X$, there is open $C\subseteq B$ so that for any $D\subseteq C$, the set $C\setminus VD$ is nowhere dense.
\end{enumerate}
\end{theorem}

It is proven in \cite{AKL} that comeager orbits push forward; namely, if $X$ is a minimal $G$-flow, $x\in X$ is a point whose orbit is generic, and if $\pi: X\rightarrow Y$ is a surjective $G$-map, then $\pi(x)$ has generic orbit in $Y$. Theorem \ref{MainTheorem} then becomes equivalent to the following: whenever $G$ is a Polish group and $X$ is a minimal metrizable flow with all orbits meager, then $G$ must admit some minimal, non-metrizable flow. Remarkably, neither \cite{BYMT} nor \cite{Z} prove Theorem \ref{MainTheorem} in this direct fashion.

We provide a direct proof of Theorem \ref{MainTheorem}. For any topological group $G$ and any $G$-flow $X$, we construct a new $G$-flow denoted $S_G(X)$. We then show that if $X$ is minimal, then so is $S_G(X)$. Lastly, if $G$ is Polish and $X$ is metrizable and has all orbits meager, we use Theorem \ref{RosendalThm} to show that $S_G(X)$ is non-metrizable.

After providing our new proof of \ref{MainTheorem}, we investigate the flow $S_G(X)$ in more detail. For any $G$-flow $X$, there is a natural map $\pi_X: S_G(X)\rightarrow X$. When $X$ is minimal, we show that $\pi_X$ is the \emph{universal highly proximal extension} of $X$. The notion of a highly proximal extension was introduced by Auslander and Glasner in \cite{AG}. If $X$ and $Y$ are minimal $G$-flows, a $G$-map $\phi: Y\rightarrow X$ is \emph{highly proximal} if for any $x\in X$ and non-empty open $U\subseteq Y$, there is $g\in G$ with $g\pi^{-1}(\{x\})\subseteq U$. Auslander and Glasner prove in \cite{AG} that for every minimal $G$-flow $X$, there is a \emph{universal highly proximal extension} $\pi: \widehat{X}\rightarrow X$. This means that $\pi$ is highly proximal, and for every other highly proximal $\phi: Y\rightarrow X$, there is a $G$-map $\psi: \widehat{X}\rightarrow Y$ so that $\pi = \phi\circ \psi$. The map $\pi$ is unique up to $G$-flow isomorphism over $X$.  Our construction of the flow $S_G(X)$ provides a new construction of the universal highly proximal extension of $X$ and hints at a generalization of this notion even when $X$ is not minimal.

\subsection{Acknowledgments} 

I would like to thank Eli Glasner for many helpful discussions, including the initial suggestion that $\pi_X$ was the universal highly proximal extension of $X$. I would also like to thank Todor Tsankov for helpful discussions, and I would like to thank the Casa Matem\'atica Oaxaca for their hospitality while some of this work was being completed.

\section{The flow $S_G(X)$ and proof of Theorem \ref{MainTheorem}}

All groups and spaces will be assumed Hausdorff. In this section, fix a topological group $G$ and a $G$-flow $X$. Write $\mathcal{N}_G$ for the collection of symmetric open neighborhoods of the identity in $G$, and write $\op{X}$ for the collection of nonempty open subsets of $X$.

\begin{defin}
A \emph{near filter} is any $\mathcal{F}\subseteq \op{X}$ so that for any $A_1,...,A_k\in \mathcal{F}$ and any $U\in \mathcal{N}_G$, we have $UA_1\capdots UA_k\neq\emptyset$. A \emph{near ultrafilter} is a maximal near filter.
\end{defin}

Near ultrafilters exist by an application of Zorn's lemma. Near ultrafilters on a uniform space have been considered in \cite{AK} and \cite{B}. Two aspects of our approach are slightly different. First, the notion of nearness is not given by the natural uniform structure on the compact Hausdorff space $X$. Second, instead of working with a notion of nearness on $\mathcal{P}(X)$, we are more or less working with the regular open algebra on $X$ (see item (2) in Lemma \ref{BasicProps}). 

 Let $S_G(X)$ denote the space of near ultrafilters on $\op{X}$.
 
 \newpage

\begin{lemma}\mbox{}
\label{BasicProps}
\begin{enumerate}
\item
Let $p\in S_G(X)$, and let $A\subseteq X$ be open. If $A\not\in p$, then there is some $V\in \mathcal{N}_G$ with $VA\not\in p$.
\item
Let $A\subseteq X$ be open, and let $B_1,...,B_k\subseteq A$ be open with $B_1\cupdots B_k$ dense in $A$. If $p\in S_G(X)$ and $A\in p$, then $B_i\in p$ for some $i\leq k$.
\end{enumerate}
\end{lemma}

\begin{proof}\mbox{}
\begin{enumerate}
\item
As $A\not\in p$, find $B_1,...,B_n\in p$ and $U\in \mathcal{N}_G$ with $UA\cap UB_1\capdots UB_n = \emptyset$. Let $V\in \mathcal{N}_G$ with $VV\subseteq U$. Then $V(VA)\cap VB_1\capdots VB_n = \emptyset$.

\item
Towards a contradiction, assume $B_i\not\in p$ for each $i\leq k$. For each $i\leq k$, find $B^{i}_1,...,B^{i}_{n_i}\in p$ and a $U\in \mathcal{N}_G$ so that $UB_i\cap UB^i_1\capdots UB^i_{n_i} = \emptyset$. We can take the same $U\in \mathcal{N}_G$ for each $i\leq k$ by intersecting. Let $C = \bigcap_{i\leq k} \bigcap_{j\leq n_i} UB^i_j$. Then since $A\in p$, we have $UA\cap C \neq \emptyset$. Let $ga\in UA\cap C$, where $g\in U$ and $a\in A$. Since $UA\cap C$ is open, there is open $A'\subseteq A$ with $gA'\subseteq UA\cap C$. As $B_1\cupdots B_k$ is dense in $A$, there is some $i\leq k$ and some $b\in B_i$ with $gb\in UA\cap C$. Since $gb\in UB_i$, this is a contradiction. \qedhere
\end{enumerate}
\end{proof}

\begin{defin}
\label{TopDef}
If $A\in \op{X}$, set $N_A := \{p\in S_G(X): A\not\in p\}$. We endow $S_G(X)$ with the topology whose typical basic open neighborhood is $N_A$ for $A\in \op{X}$.
\end{defin}

\begin{prop}
\label{Topology}
The topology from Definition \ref{TopDef} is compact Hausdorff.
\end{prop}

\begin{proof}
To show that $S_G(X)$ is Hausdorff, let $p\neq q\in S_G(X)$. Find some $A\in p\setminus q$. As $A\not\in q$, find some $V\in \mathcal{N}_G$ so that $VA\not\in q$. Set $B = \mathrm{int}(X\setminus VA)$. Then $B\not\in p$. So $p\in N_B$, $q\in N_{VA}$, and $N_{VA}\cap N_B = \emptyset$.
\vspace{3 mm}

To show that $S_G(X)$ is compact, suppose $\mathcal{C}:= \{N_{A_i}: i\in I\}$ is a collection of basic open sets without a finite subcover. Then for any $i_1,...,i_k\in I$, we can find $p\in \bigcap_{j\leq k} S_G(X)\setminus N_{A_{i_j}}$, equivalently, with $A_{i_1},..., A_{i_k}\in p$. But this implies that $\{A_i: i\in I\}$ is a near filter, and can be extended to a near ultrafilter $q$. Therefore $\mathcal{C}$ is not an open cover.
\end{proof}

\begin{defin}
\label{Flow}
If $p\in S_G(X)$ and $g\in G$, we let $gp\in S_G(X)$ be defined by declaring $A\in gp$ iff $g^{-1}A\in p$ for each $A\in \op{X}$.
\end{defin}

\begin{prop}
\label{ContinuousFlow}
The action in Definition \ref{Flow} is continuous.
\end{prop}

\begin{proof}
First note that for a fixed $g\in G$, the map $p\rightarrow gp$ is continuous. So let $p_i, p\in S_G(X)$ and $g_i\in G$ with $p_i\rightarrow p$ and $g_i\rightarrow 1$. Suppose $A\not\in p$. Find $V\in \mathcal{N}_G$ with $VA\not\in p$. So eventually $VA\not\in p_i$. Also, as $g_i\rightarrow 1$, eventually we have $g_i^{-1}\in V$. Whenever $g_i^{-1}A\subseteq VA$, we must have $g_i^{-1}A\not\in p_i$. So eventually $A\not\in g_ip_i$. 
\end{proof}

Up until now, no assumptions on $G$ and $X$ have been needed. In fact, we did not even need $X$ to be compact to construct $S_G(X)$. We now begin adding extra assumptions to $G$ and $X$ to obtain stronger conclusions about $S_G(X)$.

\newpage

\begin{prop}
\label{Minimal}
Suppose $X$ is a minimal $G$-flow. Then so is $S_G(X)$.
\end{prop}

\begin{proof}
Let $p\in S_G(X)$, and let $A\in op(X)$ with $N_A\neq \emptyset$. Find some $V\in \mathcal{N}_G$ with $N_{VA}\neq\emptyset$. Then $B := \mathrm{int}(X\setminus VA)\neq \emptyset$. As $X$ is minimal, find $g_1,...,g_k$ with $X = \bigcup_{i\leq k} g_iB$. For some $i\leq k$, we must have $g_iB\in p$. Then $B\in g_i^{-1}p$, so we must have $A\not\in g_i^{-1}p$, and the orbit of $p$ is dense as desired.
\end{proof}

Before proving Theorem \ref{MainTheorem}, we need a sufficient criterion for when $S_G(X)$ is non-metrizable. 

\begin{prop}
\label{NonMetrizable}
Suppose there are $\{A_n: n< \omega\}\subseteq \op{X}$ and $V\in \mathcal{N}_G$ so that the collection $\{VA_n: n< \omega\}$ is pairwise disjoint. Then $S_G(X)$ is non-metrizable.
\end{prop}

\begin{proof}
If $S\subseteq \omega$, let $A_S = \bigcup_{n\in S} A_n$, and let $Y = \{p\in S_G(X): A_\omega\in p\}$. Then $Y\subseteq S_G(X)$ is a closed subspace. To show that $S_G(X)$ is non-metrizable, we will exhibit a continuous surjection $\pi: Y\rightarrow \beta\omega$. First note that if $S\subseteq \omega$, then $VA_S\cap VA_{\omega\setminus S} = \emptyset$. Therefore, if $p\in Y$, $p$ contains exactly one of $A_S$ or $A_{\omega\setminus S}$ for each $S\subseteq \omega$. We let $\pi: Y\rightarrow \beta \omega$ be defined so that for $S\subseteq \omega$, $S\in \pi(p)$ iff $A_S\in p$. It is immediate that $\pi$ is continuous. To see that $\pi$ is surjective, let $q\in \beta\omega$. Then $\{A_S: S\in q\}$ is a near filter; any near ultrafilter $p$ extending it is a member of $Y$ with $\pi(p) = q$. 
\end{proof}

\begin{proof}[Proof of Theorem \ref{MainTheorem}]
We now fix a Polish group $G$ and a minimal $G$-flow $X$ whose orbits are all meager. Then by Theorem \ref{RosendalThm}, there is $U\in \mathcal{N}_G$ and open $B\subseteq X$ so that for any open $C\subseteq B$, there is open $D\subseteq C$ with $C\setminus UD$ somewhere dense (since $C$ and $UD$ are open, this is the same as $C\setminus UD$ having nonempty interior). 

Let $V\in \mathcal{N}_G$ with $VV\subseteq U$. We now produce $\{A_n: n<\omega\}\subseteq \op{X}$ with $\{VA_n: n<\omega\}$ pairwise disjoint. First set $B_0 = B$. As $B_0\subseteq B$, there is $A_0\subseteq B_0$ so that $B_0\setminus UA_0$ has nonempty interior. Suppose open sets $B_0,...,B_{n-1}$ and $A_0,...,A_{n-1}$ have been produced so that $A_i\subseteq B_i$ and $\mathrm{int}(B_i\setminus UA_i)\neq \emptyset$. We continue by setting $B_n = \mathrm{int}(B_{n-1}\setminus UA_{n-1})$. As $B_n\subseteq B$, there is $A_n\subseteq B_n$ so that $B_n\setminus UA_n$ has nonempty interior. Notice that for any $m\leq n$, we also have $A_n\subseteq B_m$. It follows that if $m < n$, we have $UA_m\cap A_n = \emptyset$. This implies that $VA_m\cap VA_n = \emptyset$ as desired. We can now apply Proposition \ref{NonMetrizable} to conclude that $S_G(X)$ is not metrizable.
\end{proof}

\section{Universal highly proximal extensions}

Let $\phi: Y\rightarrow X$ be a $G$-map between minimal flows. There are several equivalent definitions which all say that $\phi$ is \emph{highly proximal}. The definition we will use here is that $\phi$ is highly proximal iff every non-empty open $B\subseteq Y$ contains a fiber $\phi^{-1}(\{x\})$ for some $x\in X$. Define the \emph{fiber image} of $B$ to be the set $\phi_{fib}(B) := \{x\in X: \phi^{-1}(\{x\})\subseteq B\}$. Notice that $\phi_{fib}(B)$ is open, and $\phi$ is highly proximal iff $\phi_{fib}(B)\neq \emptyset$ for every non-empty open $B\subseteq Y$. It follows that this definition is the same as the one given in the introduction.

Now let $X$ be a $G$-flow, and form $S_G(X)$. We define the map $\pi_X: S_G(X)\rightarrow X$ as follows. For each $p\in S_G(X)$, there is a unique $x_p\in X$ so that every neighborhood of $x_p$ is in $p$. The existence of such a point is an easy consequence of the compactness of $X$ and the second item of \ref{BasicProps}. For uniqueness, notice that if $x\neq y\in X$, we can find open $A\ni x$, $B\ni y$ and $U\in \mathcal{N}_G$ with $UA\cap UB = \emptyset$. We set $\pi_X(p) = x_p$. This map clearly respects the $G$-action. To check continuity, one can check that if $K\subseteq X$ is closed, then $\pi_X^{-1}(K) = \{p\in S_G(X): A\in p \text{ for every open } A\supseteq K\}$, and this is a closed condition.

\begin{prop}
	\label{HP}
	Let $X$ be minimal. Then the map $\pi_X: S_G(X)\rightarrow X$ is highly proximal.
\end{prop}

\begin{proof}
	By \ref{Minimal}, $S_G(X)$ is a minimal flow. So let $N_A\subseteq S_G(X)$ be a nonempty basic open neighborhood. This implies that $\mathrm{int}(X\setminus A)\neq\emptyset$. Let $x\in \mathrm{int}(X\setminus A)$. Then there are open $B\ni x$ and $U\in \mathcal{N}_G$ with $UB\cap A = \emptyset$. It follows that any $p\in S_G(X)$ containing $B$ cannot contain $A$. In particular, we have $\pi_X^{-1}(\{x\})\subseteq N_A$.
\end{proof}

\begin{theorem}
	\label{UnivHP}
	Let $X$ be minimal. Then the map $\pi_X: S_G(X)\rightarrow X$ is the universal highly proximal extension of $X$.
\end{theorem}

\begin{proof}
Fix a highly proximal extension $\phi: Y\rightarrow X$. For each $y\in Y$, let $\mathcal{F}_y := \{\phi_{fib}(B): B\ni y \text{ open}\}$. Then $\mathcal{F}_y\subseteq \op{X}$ is a filter of open sets, so in particular it is a near filter. We will show that for each $p\in S_G(X)$, there is a unique $y\in Y$ with $\mathcal{F}_y\subseteq p$. This will define the map $\psi: S_G(X)\rightarrow Y$.

We first show that for each $p\in S_G(X)$, there is at least one such $y\in Y$. To the contrary, suppose for each $y\in Y$, there were $B_y\ni y$ open so that $\phi_{fib}(B_y)\not\in p$. Find $y_1,...,y_k$ so that $\{B_{y_1},...,B_{y_k}\}$ is a finite subcover. Let $A_i = \phi_{fib}(B_{y_i})$. Each $A_i$ is open, so we will reach a contradiction once we show that $\bigcup_{i\leq k} A_i$ is dense. Let $A\subseteq X$ be open. Then $C := B_{y_i}\cap \phi^{-1}(A)\neq \emptyset$ for some $i\leq k$. As $C$ is open, $\phi_{fib}(C)\neq \emptyset$, and $\phi_{fib}(C)\subseteq A\cap A_i$.

Now we consider uniqueness. Let $p\in S_G(X)$, and consider $y\neq z\in Y$. Find open $B\ni y$ and $C\ni z$ and some $V\in \mathcal{N}_G$ so that $VB\cap VC = \emptyset$. It follows that $\phi_{fib}(VB)\cap \phi_{fib}(VC) = \emptyset$. Now notice that $V\phi_{fib}(B)\subseteq \phi_{fib}(VB)$, and likewise for $C$. Hence $p$ cannot contain both $\mathcal{F}_y$ and $\mathcal{F}_z$.

The map $\psi$ clearly respects the $G$-action and satisfies $\pi_X = \phi\circ \psi$. To show continuity, let $K\subseteq Y$ be closed. Let $\mathcal{F}_K := \{\phi_{fib}(B): B\supseteq K \text{ open}\}$. We will show that $\psi(p)\in K$ iff $\mathcal{F}_K\subseteq p$. From this it follows that $\psi^{-1}(K)$ is closed. One direction is clear. For the other, suppose $\psi(p) = y\not\in K$. Find open sets $B\ni y$, $C\supseteq K$, and $V\in \mathcal{N}_G$ with $VB\cap VC = \emptyset$. As in the proof of uniqueness, $p$ cannot contain both $\mathcal{F}_y$ and $\mathcal{F}_K$.
\end{proof}

By combining the main results of the previous two sections, we obtain the following.

\begin{cor}
	Let $G$ be a Polish group, and let $X$ be a minimal, metrizable $G$-flow with all orbits meager. Then the universal highly proximal extension of $X$ is non-metrizable.
\end{cor}

\end{document}